\documentclass[11pt,british]{amsart}

\usepackage{babel,eucal,url,amssymb,%stmaryrd,%booktabs,%
enumerate,amscd,%paralist
}

\usepackage[pagebackref]{hyperref}

\theoremstyle{plain}
\newtheorem{thm}{Theorem}[section]
\newtheorem{lem}[thm]{Lemma}
\newtheorem{pro}[thm]{Proposition}
\newtheorem{co}[thm]{Corollary}

\theoremstyle{definition}
\newtheorem{defn}[thm]{Definition}

\theoremstyle{remark}

\newtheorem{example}[thm]{Example}

\newcommand{\Gtwo}{\ifmmode{{\rm G}_2}\else{${\rm G}_2$}\fi}

\date{\today}

\begin{document}

\title[]
 {On the geometry of trans-para-Sasakian manifolds}

\author[S. Zamkovoy]{Simeon Zamkovoy}

\address{
University of Sofia "St. Kl. Ohridski"\\
Faculty of Mathematics and Informatics\\
Blvd. James Bourchier 5\\
1164 Sofia, Bulgaria}
\email{zamkovoy@fmi.uni-sofia.bg}

%\author[G. Nakova]{Galia Nakova}

%\address{
%University of Veliko Tarnovo "St. Cyril and St. Methodius" \\ Faculty of Mathematics and Informatics\\   Department of Algebra and Geometry
%\\ 2 Teodosii Tarnovski Str. \\ Veliko Tarnovo 5003\\ Bulgaria}
%\email{gnakova@gmail.com}

\subjclass{53D15}

\keywords{trans-para-Sasakian manifolds, 3-dimensional trans-para-Sasakian manifolds, $\xi-$sectional curvature}

%\date{January 1, 2004}
%----------additions
%\dedicatory{To my boss}
%%% ----------------------------------------------------------------------

\begin{abstract}
In this paper, we introduce the trans-para-Sasakian manifolds and we study their geometry. These manifolds are an analogue of the trans-Sasakian manifolds in the Riemannian geometry. We shall investigate many curvature properties of these manifolds and we shall give many conditions under which the manifolds are either $\eta-$Einstein or Einstein manifolds.
\end{abstract}

\newcommand{\g}{\mathfrak{g}}
\newcommand{\s}{\mathfrak{S}}
\newcommand{\D}{\mathcal{D}}
\newcommand{\F}{\mathcal{F}}
\newcommand{\R}{\mathbb{R}}
\newcommand{\K}{\mathbb{K}}
\newcommand{\U}{\mathbb{U}}
\newcommand{\diag}{\mathrm{diag}}
\newcommand{\End}{\mathrm{End}}
\newcommand{\im}{\mathrm{Im}}
\newcommand{\id}{\mathrm{id}}
\newcommand{\Hom}{\mathrm{Hom}}

\newcommand{\Rad}{\mathrm{Rad}}
\newcommand{\rank}{\mathrm{rank}}
\newcommand{\const}{\mathrm{const}}
\newcommand{\tr}{{\rm tr}}
\newcommand{\ltr}{\mathrm{ltr}}
\newcommand{\codim}{\mathrm{codim}}
\newcommand{\Ker}{\mathrm{Ker}}

\newcommand{\thmref}[1]{Theorem~\ref{#1}}
\newcommand{\propref}[1]{Proposition~\ref{#1}}
\newcommand{\corref}[1]{Corollary~\ref{#1}}
\newcommand{\secref}[1]{\S\ref{#1}}
\newcommand{\lemref}[1]{Lemma~\ref{#1}}
\newcommand{\dfnref}[1]{Definition~\ref{#1}}
%\newcommand{\eqref}[1]{(\ref{#1})}

%\frenchspacing

\newcommand{\ee}{\end{equation}}
\newcommand{\be}[1]{\begin{equation}\label{#1}}

%%% ----------------------------------------------------------------------
\maketitle

\section{Introduction}\label{sec-1}
In Grey-Hervella classification of almost Hermitian manifolds (see \cite{GH}), there appears a class, $\mathcal{W}_4$, of Hermitian manifolds which are closely related to locally conformal $K\ddot{a}hler$
manifolds. An almost contact structure on a manifold $M$ is called a \emph{trans-Sasakian structure} (see \cite{O}) if the product manifold $M\times \mathbb{R}$ belongs to the class $\mathcal{W}_4$. The class
$\mathcal{C}_6\bigoplus\mathcal{C}_5$ (see \cite{M}, \cite{MC}) coincides with the class of trans-Sasakian structures of type $(\alpha,\beta)$. In fact, in (see \cite{MC}), local nature of the two subclasses, namely the  $\mathcal{C}_5$ and the $\mathcal{C}_6$ structures, of trans-Sasakian structures are characterized completely. We note the that trans-Sasakian structures of type $(0,0)$, $(0,\beta)$ and $(\alpha,0)$ are cosympletic (see \cite{B1}), $\beta-$Kenmotsu (see \cite{JV}) and $\alpha-$Sasakian (see \cite{JV}), respectively. We consider the trans-para-Sasakian manifolds as an analogue of the trans-Sasakian manifolds. A trans-para-Sasakian manifold is a trans-para-Sasakian structure of type $(\alpha,\beta)$, where $\alpha$ and $\beta$ are smooth functions. The trans-para-Sasakian manifolds of types $(\alpha,\beta)$, and are respecively the para-cosympletic, para-Sasakian (in case $\alpha=1$, these are just the para-Sasakian manifolds; in case $\alpha=-1$, these are the quasi-para-Sasakian manifolds, see \cite{Z2}) and  para-Kenmotsu (for the case $\beta=1$  see \cite{Z3}). In the second section, we give the formal definition of trans-para-Sasakian manifolds of type $(\alpha,\beta)$ and we prove some basic properties. We give an example for a 3-dimensional trans-para-Sasakian manifold. In the last section, we investigate the curvature properties of the trans-para-Sasakian manifolds. Further, we find many conditions under which the manifolds are either $\eta-$Einstein or Einstein manifolds.
\section{Preliminaries}\label{sec-2}
A (2n+1)-dimensional smooth manifold $M^{(2n+1)}$
has an \emph{almost paracontact structure}
$(\varphi,\xi,\eta)$
if it admits a tensor field
$\varphi$ of type $(1,1)$, a vector field $\xi$ and a 1-form
$\eta$ satisfying the  following compatibility conditions % are satisfied:
\begin{eqnarray}
  \label{f82}
    & &
    \begin{array}{cl}
     (i)   & \varphi(\xi)=0,\quad \eta \circ \varphi=0,\quad
     \\[5pt]
     (ii)  & \eta (\xi)=1 \quad \varphi^2 = id - \eta \otimes \xi,
     \\[5pt]
     (iii) & \textrm{distribution $\mathbb {D}: p \in M \longrightarrow \mathbb {D}_p\subset T_pM:$}
     \\[1pt]
     & \textrm{$\mathbb D_p=Ker \eta=\{X\in T_pM: \eta (X)=0\}$ is called {\it paracontact}}
     \\[1pt]
     & \textrm{{\it distribution} generated by $\eta$.}
    \end{array}
\end{eqnarray}

The tensor field $\varphi $ induces an almost paracomplex structure \cite{KW} on each
fibre on $\mathbb D$ and $(\mathbb D, \varphi , g_{\vert \mathbb D})$ is a $2n$-dimensional
almost paracomplex distribution. Since $g$ is non-degenerate metric on $M$ and $\xi $ is non-isotropic,
the paracontact distribution $\mathbb D$ is non-degenerate.

An immediate consequence of the definition of the almost
paracontact structure is that the endomorphism $\varphi$ has rank
$2n$,
%an almost paracontact structure  $(\varphi,\xi,\eta)$ has the properties
$\varphi \xi=0$ and $\eta \circ \varphi=0$, (see \cite{B1,B2}
for the almost contact case).

%Following  \cite{B1,B2},

%\begin{thm}\label{t1}
%Suppose $M^{(2n+1)}$ has a $(\varphi,\xi,\eta)$-structure. Then
%$\varphi \xi=0$ and $\eta \circ \varphi=0$. Moreover the
%endomorphism $\varphi$ has rank $2n$.
%\end{thm}
If a manifold $M^{(2n+1)}$ with $(\varphi,\xi,\eta)$-structure
admits a pseudo-Riemannian metric $g$ such that
\begin{equation}\label{con}
g(\varphi X,\varphi Y)=-g(X,Y)+\eta (X)\eta (Y),
\end{equation}
then we say that $M^{(2n+1)}$ has an almost paracontact metric structure and
$g$ is called \emph{compatible}. Any compatible metric $g$ with a given almost paracontact
structure is necessarily of signature $(n+1,n)$.

Note that setting $Y=\xi$, we have
%immediately that
$\eta(X)=g(X,\xi).$
%As in the case of an almost paracomplex structure the existence of
%the compatible metric is easy.

Further, any almost paracontact structure admits a compatible metric.
\begin{defn}
If $g(X,\varphi Y)=d\eta(X,Y)$ (where
$d\eta(X,Y)=\frac12(X\eta(Y)-Y\eta(X)-\eta([X,Y])$ then $\eta$ is
a paracontact form and the almost paracontact metric manifold
$(M,\varphi,\eta,\xi,g)$ is said to be a $\emph{paracontact metric
manifold}$.
\end{defn}
A paracontact metric manifold for which $\xi$ is Killing is called a $K-\emph{paracontact}$ $\emph{manifold}$. A paracontact structure on $M^{(2n+1)}$ naturally gives rise to an almost paracomplex structure on the product $M^{(2n+1)}\times\Re$. If this almost paracomplex structure is integrable, then the given paracontact metric manifold is said to be a $\emph{para-Sasakian}$. Equivalently, (see \cite{Z}) a paracontact metric manifold is a para-Sasakian if and only if
\begin{equation}\label{8}
(\nabla_{X}\varphi)Y=-g(X,Y)\xi+\eta(Y)X,
\end{equation}
 for all vector fields $X$ and $Y$ (where
$\nabla$ is the Livi-Civita connection of $g$).
\begin{defn}\label{d1}
If $(\nabla_X\varphi)Y=\alpha(-g(X,Y)\xi+\eta(Y)X)+\beta(g(X,\varphi Y)\xi+\eta(Y)\varphi X),$ then the manifold
$(M^{(2n+1)},\varphi,\eta,\xi,g)$ is said to be a $\emph{trans-para-Sasakian manifold}$.
\end{defn}
From $Definition~\ref{d1}$ we have
\begin{equation}\label{e0}
\nabla_{X}\xi=-\alpha\varphi X-\beta(X-\eta(X)\xi).
\end{equation}
\begin{defn}
A $(2n+1)$-dimensional almost paracontact metric manifold is called
\item {\it normal} if $N(X,Y)-2d\eta (X,Y)\xi = 0$, where
$N(X,Y)=\varphi ^2[X,Y]+[\varphi X,\varphi Y]-\varphi [\varphi X,Y]-\varphi [X,\varphi Y]$
is the Nijenhuis torsion tensor of $\varphi $ (see \cite{Z}).
\end{defn}
Denoting by $\pounds$ the Lie differentiation of $g$, we see
\begin{pro}\label{pro1}
Let $(M^{(2n+1)},\varphi,\eta,\xi,g)$ be a trans-para-Sasakian manifold. Then we have
\begin{equation}\label{e1}
(\nabla_X\eta)Y=\alpha g(X,\varphi Y)-\beta(g(X,Y)-\eta(X)\eta(Y)),
\end{equation}
\begin{equation}\label{e1.1}
d\eta(X,Y)=\alpha g(X,\varphi Y),
\end{equation}
\begin{equation}\label{e2}
(\pounds_{\xi}g)(X,Y)=-2\beta(g(X,Y)-\eta(X)\eta(Y)),
\end{equation}
\begin{equation}\label{e3}
\pounds_{\xi}\varphi=0,
\end{equation}
\begin{equation}\label{e4}
\pounds_{\xi}\eta=0,
\end{equation}
where $X,Y\in T_pM.$
\end{pro}
Since the proof of $Proposition~\ref{pro1}$ follows by routine calculation, we shall omit it.

From $Proposition~\ref{pro1}$ we see that $(M^{(2n+1)},\varphi,\eta,\xi,g)$ is normal.

\begin{example}
Let us consider the 3-dimensional manifold $M^3=\{(x,y,z):(x,y,z)\in \Re^{3}_{1}\},z\neq 0$, where $(x,y,z)$ are the standard coordinates in $\Re^{3}_{1}$.We choose the vector fields
$$E_1=e^z(\frac{\partial}{\partial x}+y\frac{\partial}{\partial z}). \quad E_2=e^z\frac{\partial}{\partial y}, \quad E_3=\frac{\partial}{\partial z},$$
which are linearly independent at each point of $M$. We define an almost paracontact structure $(\varphi ,\xi ,\eta)$ and a pseudo-Riemannian metric $g$ in the following way:
\[
\begin{array}{llll}
\varphi E_1=E_2 , \quad \varphi E_2=E_1 , \quad \varphi E_3=0 \\
\xi =E_3 , \quad \eta (E_3)=1 , \quad \eta (E_1)=\eta (E_2)=0 , \\
g(E_1,E_1)=g(E_3,E_3)=-g(E_2,E_2)=1 ,\\
\quad g(E_i,E_j)=0, \quad i\neq j \in \{1,2,3\}.
\end{array}
\]
By the definition of Lie bracket, we have
\[
\begin{array}{llll}
[E_1,E_2]=ye^zE_2-e^{2z}E_3, \quad [E_2,E_3]=-E_2, \quad [E_1,E_3]=-E_3.
\end{array}
\]
Then $(M,\varphi ,\xi ,\eta ,g)$ is a 3-dimensional almost paracontact manifold. The Koszul equality becomes
\[
\begin{array}{llll}
\nabla_{E_1}E_1=E_3 , \quad \nabla_{E_1}E_2=-\frac{1}{2}e^{2z}E_3 , \quad \nabla_{E_1}E_3=-E_1-\frac{1}{2}e^{2z}E_2, \\
\nabla_{E_2}E_1=-ye^zE_2+\frac{1}{2}e^{2z}E_3  , \quad \nabla_{E_2}E_2=-ye^zE_1-E_3, \quad \nabla_{E_2}E_3=-\frac{1}{2}e^{2z}E_1-E_2, \\
\nabla_{E_3}E_1=-\frac{1}{2}e^{2z}E_2  , \quad \nabla_{E_3}E_2=-\frac{1}{2}e^{2z}E_1, \quad \nabla_{E_3}E_3=0.   \\
\end{array}
\]
We have $\nabla_{E_1}\xi=-\alpha\varphi E_1-\beta E_2, \quad \nabla_{E_2}\xi=-\alpha\varphi E_2-\beta E_2, \quad \nabla_{\xi}\xi=0$ for $E_3=\xi$, where
$\alpha=\frac{1}{2}e^{2z}$ and $\beta=1$.

Again, by virtue of \eqref{e1} and $(\nabla_X\eta)Y=X(\eta(Y))-\eta(\nabla_XY)$ we obtain
$$(\nabla_{E_1}\eta)E_1=-\beta=-1, \quad (\nabla_{E_2}\eta)E_1=-\alpha=-\frac{1}{2}e^{2z}, \quad (\nabla_{E_3}\eta)E_1=0.$$
Thus from above the calculation the condition \eqref{e0} and \eqref{e1} are satisfied and the structure $(\varphi ,\xi ,\eta ,g)$ is a trans-para-Sasakian structure of type $(\alpha,\beta)$,
where $\alpha=\frac{1}{2}e^{2z}$ and $\beta=1$. Consequently $(M^3,\varphi ,\xi ,\eta ,g)$ is a trans-para-Sasakian manifold.
\end{example}
Finally, the sectional curvature $K(\xi,X)=\epsilon_XR(X,\xi,\xi,X)$, where $|X|=\epsilon_X=\pm 1$, of a plane section spanned by $\xi$ and the vector $X$ orthogonal to $\xi$ is called $\emph{$\xi$-sectional curvature}$, where denoting by $R$ the curvature tensor of $\nabla$.
\section{Some curvatureb properties of trans-para-Sasakian manifolds}\label{sec-3}
We begin with the following Lemma.
\begin{lem}\label{lem1}
Let $(M^{(2n+1)},\varphi,\eta,\xi,g)$ be a trans-para-Sasakian manifold. Then we have
\begin{equation}\label{e5}
R(X,Y)\xi=-(\alpha^2+\beta^2)(\eta(Y)X-\eta(X)Y)-2\alpha\beta(\eta(Y)\varphi X-\eta(X)\varphi Y)-
\end{equation}
$$-X(\alpha)\varphi Y+Y(\alpha)\varphi X+Y(\beta)\varphi^2X-X(\beta)\varphi^2Y.$$
\end{lem}
\begin{proof}
Using $Definition~\ref{d1}$, we obtain
$$\nabla_X\nabla_Y\xi=\nabla_X(-\alpha\varphi Y-\beta(Y-\eta(Y)\xi)=$$
$$=-X(\alpha)\varphi Y-\alpha\nabla_X{\varphi Y}-X(\beta)\varphi^2Y-\beta\nabla_XY-\beta(X\eta(Y))\xi-$$
$$-\alpha\beta\eta(Y)\varphi X-\beta^2\eta(Y)X+\beta^2\eta(X)\eta(Y)\xi,$$
From here and \eqref{e0}, we get
$$R(X,Y)\xi=\nabla_X\nabla_Y\xi-\nabla_Y\nabla_X\xi-\nabla_{[X,Y]}\xi=$$
$$=-X(\alpha)\varphi Y+Y(\alpha)\varphi X-\alpha((\nabla_X\varphi)Y-(\nabla_Y\varphi)X)-$$
$$-X(\beta)\varphi^{2}Y+Y(\beta)\varphi^{2}X+\beta((\nabla_X\eta)Y-(\nabla_Y\eta)X)\xi-$$
$$-\alpha\beta(\eta(Y)\varphi X-\eta(X)\varphi Y)-\beta^2(\eta(Y)X-\eta(X)Y),$$
which in view of $Definition~\ref{d1}$ and \eqref{e1} gives \eqref{e5}.
\end{proof}
$Lemma~\ref{lem1}$ yields the following
\begin{pro}
If $(M^{(2n+1)},\varphi,\eta,\xi,g)$ is a trans-para-Sasakian manifold, then it is of $\xi-$sectional curvature $K(\xi,X)=-\epsilon_X(\alpha^2+\beta^2-\xi(\beta))$.

\end{pro}
In a trans-para-Sasakian manifolds the functions $\alpha$ and $\beta$ can not be arbitrary. This fact is shown in the following
\begin{thm}\label{thm1}
In trans-para-Sasakian manifold, we have
\begin{equation}\label{e6}
R(\xi,X)\xi=(\alpha^2+\beta^2-\xi(\beta))(X-\eta(X)\xi),
\end{equation}
\begin{equation}\label{e7}
2\alpha\beta-\xi(\alpha)=0.
\end{equation}
\end{thm}
\begin{proof}
Using \eqref{e5} in $R(\xi,Z,X,Y)=R(X,Y,\xi,Z)$, we get
\begin{equation}\label{e8}
R(\xi,Z)X=-(\alpha^2+\beta^2)(g(X,Z)-\eta(X)Z)-2\alpha\beta(g(\varphi X,Z)\xi+\eta(X)\varphi Z)+
\end{equation}
$$+X(\alpha)\varphi Z+g(\varphi X,Z)grad\alpha-X(\beta)(Z-\eta(Z)\xi)-g(\varphi X,\varphi Z)grad\beta.$$
From \eqref{e5}, we get
$$R(\xi,X)\xi=(\alpha^2+\beta^2-\xi(\beta))(X-\eta(X)\xi)+(2\alpha\beta-\xi(\alpha))\varphi Y,$$
while gives us \eqref{e5}
$$R(\xi,X)\xi=(\alpha^2+\beta^2-\xi(\beta))(X-\eta(X)\xi)-(2\alpha\beta-\xi(\alpha))\varphi Y.$$
The above two equations provide \eqref{e6} and \eqref{e7}.
\end{proof}
From $Lemma~\ref{lem1}$, we have the following
\begin{pro}\label{pro2}
In a $(2n+1)-$dimensional tras-para-Sasakian manifold, we have
\begin{equation}\label{e9}
Ric(X,\xi)=-(2n(\alpha^2+\beta^2)-\xi(\beta))\eta(X)+(2n-1)X(\beta)-\varphi X(\alpha),
\end{equation}
\begin{equation}\label{e10}
Q\xi=-(2n(\alpha^2+\beta^2)-\xi(\beta))\xi+(2n-1)grad\beta+\varphi(grad\alpha),
\end{equation}
where $Ric$ is the Ricci tensor and $Q$ is the Ricci operator given by
\begin{equation}\label{e11}
Ric(X,Y)=g(QX,Y).
\end{equation}
\end{pro}
\begin{co}\label{co1}
If in a $(2n+1)-$dimensional trans-para-Sasakian manifold we have $\varphi(grad\alpha)=-(2n-1)grad\beta$, then
$$\xi(\beta)=g(\xi,grad\beta)=-\frac{1}{2n-1}g(\xi,\varphi(grad\alpha))=0,$$
and hence
\begin{equation}\label{e12}
Ric(X,\xi)=-2n(\alpha^2+\beta^2)\eta(X),
\end{equation}
\begin{equation}\label{e13}
Q\xi=-2n(\alpha^2+\beta^2)\xi.
\end{equation}
\end{co}

From here on, we shall assume that  $\varphi(grad\alpha)=-(2n-1)grad\beta$.

The Weyl-projective curvature tensor $P$ is defined as
\begin{equation}\label{e14}
P(X,Y)Z=R(X,Y)Z-\frac{1}{2n}(Ric(Y,Z)X-Ric(X,Z)Y).
\end{equation}
Hence we can state the following
\begin{thm}\label{thm2}
A Weyl projectively flat trans-para-Sasakian manifold is an Einstein manifold.
\end{thm}
\begin{proof}
Suppose that $P=0$. Then from equation \eqref{e14}, we have
\begin{equation}\label{e15}
R(X,Y)Z=\frac{1}{2n}(Ric(Y,Z)X-Ric(X,Z)Y).
\end{equation}
From \eqref{e15}, we obtain
\begin{equation}\label{e16}
R(X,Y,Z,W)=\frac{1}{2n}(Ric(Y,Z)g(X,W)-Ric(X,Z)g(Y,W)).
\end{equation}
Putting $W=\xi$ in \eqref{e16}, we get
\begin{equation}\label{e17}
\eta(R(X,Y)Z)=\frac{1}{2n}(Ric(Y,Z)\eta(X)-Ric(X,Z)\eta(Y)).
\end{equation}
Again taking $X=\xi$, and using \eqref{e5} and \eqref{e12}, we get
\begin{equation}\label{e18}
Ric(X,Y)=-2n(\alpha^2+\beta^2)g(X,Y).
\end{equation}
\end{proof}
\begin{thm}\label{thm3}
A trans-para-Sasakian manifold satisfying $R(X,Y)P=0$ is an Einstein manifold and also it is a manifold of scalar curvature $scal=-2n(2n+1)(\alpha^2+\beta^2)$.
\end{thm}
\begin{proof}
Using \eqref{e5} and \eqref{e12} in \eqref{e14}, we get
\begin{equation}\label{e19}
\eta(P(X,Y)\xi)=0
\end{equation}
and
\begin{equation}\label{e20}
\eta(P(\xi,Y)Z)=-(\alpha^2+\beta^2)g(Y,Z)-\frac{1}{2n}Ric(Y,Z)
\end{equation}
Now,
$$(R(X,Y)P(U,V)Z=R(X,Y)P(U,V)Z-P(R(X,Y)U,V)Z-P(U,R(X,Y)V)Z-$$
$$-P(U,V)R(X,Y)Z.$$
By assumption  $R(X,Y)P=0$, so we have
\begin{equation}\label{e21}
R(X,Y)P(U,V)Z-P(R(X,Y)U,V)Z-P(U,R(X,Y)V)Z-
\end{equation}
$$-P(U,V)R(X,Y)Z=0.$$
Therefore
$$g(R(\xi,Y)P(U,V)Z,\xi)-g(P(R(\xi,Y)U,V)Z,\xi)-g(P(U,R(\xi,Y)V)Z,\xi)-$$
$$-g(P(U,V)R(\xi,Y)Z,\xi)=0.$$
From this, it follows that,
\begin{equation}\label{e22}
-P(U,V,Z,Y)+\eta(Y)\eta(P(U,V)Z)-\eta(U)\eta(P(Y,V)Z)+
\end{equation}
$$+g(Y,U)\eta(P(\xi,V)Z)-\eta(V)\eta(P(U,Y)Z)+g(Y,V)\eta(P(U,\xi)Z)-$$
$$-\eta(Z)\eta(P(U,V)Y)=0.$$
Let $\{e_i\}$, $i=1,...,2n+1$ be an orthonormal basis. Then summing up for $1\leq i \leq 2n+1$ of the relation \eqref{e22} for $Y=U=e_i$ yields
\begin{equation}\label{e23}
2n\eta(P(\xi,V)Z)+\eta(Z)P(V,e_i,e_i,\xi)=0.
\end{equation}
From \eqref{e20}, we have
\begin{equation}\label{e24}
Ric(V,Z)=-2n(\alpha^2+\beta^2)g(Y,Z)-((2n+1)(\alpha^2+\beta^2)+\frac{scal}{2n}).
\end{equation}
Taking $Z=\xi$ in \eqref{e24} and using \eqref{e12} we obtain
\begin{equation}
scal=-2n(2n+1)(\alpha^2+\beta^2)\quad and\quad Ric(V,Z)=-2n(\alpha^2+\beta^2)g(Y,Z)
\end{equation}
\end{proof}
The Weyl-conformal tensor $C$ is defined by
\begin{equation}\label{e25}
C(X,Y)Z=R(X,Y)Z-\frac{1}{2n-1}(g(Y,Z)QX-g(X,Z)QY+Ric(Y,Z)X-
\end{equation}
$$-Ric(X,Z)Y)+\frac{scal}{2n(2n-1)}(g(Y,Z)X-g(X,Z)Y).$$
We have the following
\begin{thm}\label{thm4}
A conformally flat trans-para-Sasakian manifold is an $\eta-$Einstein manifold.
\end{thm}
\begin{proof}
Suppose that $C=0$. Then from \eqref{e25}, we get
\begin{equation}\label{e26}
R(X,Y)Z=\frac{1}{2n-1}(g(Y,Z)QX-g(X,Z)QY+Ric(Y,Z)X-
\end{equation}
$$-Ric(X,Z)Y)-\frac{scal}{2n(2n-1)}(g(Y,Z)X-g(X,Z)Y).$$
From the identity \eqref{e26}, we have
\begin{equation}\label{e27}
\eta(R(X,Y)Z)=\frac{1}{2n-1}(g(Y,Z)Ric(X,\xi)-g(X,Z)Ric(Y,\xi)+\eta(X)Ric(Y,Z)-
\end{equation}
$$-\eta(Y)Ric(X,Z))-\frac{scal}{2n(2n-1)}(g(Y,Z)\eta(X)-g(X,Z)\eta(Y)).$$
Again taking $X=\xi$ in \eqref{e27}, and using \eqref{e5} and \eqref{e12} we get
\begin{equation}\label{e28}
Ric(X,Y)=((\alpha^2+\beta^2)+\frac{scal}{2n})g(Y,Z)-((2n+1)(\alpha^2+\beta^2)+\frac{scal}{2n})\eta(X)\eta(Y).
\end{equation}
\end{proof}
\begin{thm}\label{thm5}
A trans-para-Sasakian manifold satisfying $R(X,Y)C=0$ is an $\eta-$Einstein manifold.
\end{thm}
\begin{proof}
From identity \eqref{e25}, we have $\eta(C(X,Y)\xi)=0$ and
\begin{equation}\label{e29}
\eta(C(\xi,Y)Z)=\frac{1}{2n-1}((\alpha^2+\beta^2)+\frac{scal}{2n})(g(Y,Z)-\eta(Y)\eta(Z))-
\end{equation}
$$-\frac{1}{2n-1}(Ric(Y,Z)+2n(\alpha^2+\beta^2)\eta(Y)\eta(Z)).$$
Now,
$$(R(X,Y)C(U,V)Z=R(X,Y)C(U,V)Z-C(R(X,Y)U,V)Z-C(U,R(X,Y)V)Z-$$
$$-C(U,V)R(X,Y)Z.$$
By assumption  $R(X,Y)C=0$, so we have
\begin{equation}\label{e30}
R(X,Y)C(U,V)Z-C(R(X,Y)U,V)Z-C(U,R(X,Y)V)Z-
\end{equation}
$$-C(U,V)R(X,Y)Z=0.$$
Therefore
$$g(R(\xi,Y)C(U,V)Z,\xi)-g(C(R(\xi,Y)U,V)Z,\xi)-g(C(U,R(\xi,Y)V)Z,\xi)-$$
$$-g(C(U,V)R(\xi,Y)Z,\xi)=0.$$
From this, it follows that,
\begin{equation}\label{e31}
-C(U,V,Z,Y)+\eta(Y)\eta(C(U,V)Z)-\eta(U)\eta(C(Y,V)Z)+
\end{equation}
$$+g(Y,U)\eta(C(\xi,V)Z)-\eta(V)\eta(C(U,Y)Z)+g(Y,V)\eta(C(U,\xi)Z)-$$
$$-\eta(Z)\eta(C(U,V)Y)=0.$$
Let $\{e_i\}$, $i=1,...,2n+1$ be an orthonormal basis. Then summing up for $1\leq i \leq 2n+1$ of the relation \eqref{e31} for $Y=U=e_i$ yields
\begin{equation}\label{e32}
\eta(C(\xi,V)Z)=0.
\end{equation}
From \eqref{e29}, we have
\begin{equation}\label{e33}
Ric(Y,Z)=(\frac{scal}{2n}+(\alpha^2+\beta^2))g(Y,Z)-((2n+1)(\alpha^2+\beta^2)+\frac{scal}{2n})\eta(Y)\eta(Z).
\end{equation}
\end{proof}
The concicular curvature tensor $\overline{C}$ is defined by
\begin{equation}\label{e34}
\overline{C}(X,Y)Z=R(X,Y)Z-\frac{scal}{2n(2n+1)}(g(Y,Z)X-g(X,Z)Y).
\end{equation}
We have the following
\begin{thm}\label{thm6}
A trans-para-Sasakian manifold satisfying $R(X,Y)\overline{C}=0$ is an Einstein manifold and a manifold of scalar curvature $scal=-2n(2n-1)(\alpha^2+\beta^2)$.
\end{thm}
\begin{proof}
From equality \eqref{e34}, we have $\eta(\overline{C}(X,Y)\xi)=0$ and
\begin{equation}\label{e35}
\eta(\overline{C}(\xi,Y)Z)=(-\frac{scal}{2n(2n+1)}+(\alpha^2+\beta^2))(g(Y,Z)-\eta(Y)\eta(Z)).
\end{equation}
Now,
$$(R(X,Y)\overline{C}(U,V)Z=R(X,Y)\overline{C}(U,V)Z-\overline{C}(R(X,Y)U,V)Z-\overline{C}(U,R(X,Y)V)Z-$$
$$-\overline{C}(U,V)R(X,Y)Z.$$
By assumption $R(X,Y)\overline{C}=0$, so we have
\begin{equation}\label{e36}
R(X,Y)\overline{C}(U,V)Z-\overline{C}(R(X,Y)U,V)Z-\overline{C}(U,R(X,Y)V)Z-
\end{equation}
$$-\overline{C}(U,V)R(X,Y)Z=0.$$
Therefore
$$g(R(\xi,Y)\overline{C}(U,V)Z,\xi)-g(\overline{C}(R(\xi,Y)U,V)Z,\xi)-g(\overline{C}(U,R(\xi,Y)V)Z,\xi)-$$
$$-g(\overline{C}(U,V)R(\xi,Y)Z,\xi)=0.$$
From this, it follows that,
\begin{equation}\label{e37}
-\overline{C}(U,V,Z,Y)+\eta(Y)\eta(\overline{C}(U,V)Z)-\eta(U)\eta(\overline{C}(Y,V)Z)+
\end{equation}
$$+g(Y,U)\eta(\overline{C}(\xi,V)Z)-\eta(V)\eta(\overline{C}(U,Y)Z)+g(Y,V)\eta(\overline{C}(U,\xi)Z)-$$
$$-\eta(Z)\eta(\overline{C}(U,V)Y)=0.$$
Let $\{e_i\}$, $i=1,...,2n+1$ be an orthonormal basis. Then summing up for $1\leq i \leq 2n+1$ of the relation \eqref{e37} for $Y=U=e_i$ yields
\begin{equation}\label{e38}
-Ric(V,Z)+\frac{scal}{2n+1}g(V,Z)-2n(\alpha^2+\beta^2)(g(V,Z)-\eta(V)\eta(Z))-
\end{equation}
$$-\frac{scal}{2n+1}(g(V,Z)-\eta(V)\eta(Z))+\eta(Z)Ric(V,\xi)-\frac{scal}{2n+1}\eta(V)\eta(Z).$$
Using \eqref{e12} in \eqref{e38}, we have
\begin{equation}\label{e39}
Ric(Y,Z)=-2n(\alpha^2+\beta^2)g(Y,Z)
\end{equation}
and $scal=-2n(2n-1)(\alpha^2+\beta^2).$
\end{proof}
The projective Ricci tensor is defined by
\begin{equation}\label{e40}
\widetilde{P}(X,Y)=\frac{(2n+1)}{2n}Ric(X,Y)-\frac{scal}{2n}g(X,Y).
\end{equation}
We have the following
\begin{thm}\label{thm7}
A trans-para-Sasakian manifold satisfying $R(X,Y)\widetilde{P}=0$ is an Einstein manifold and a manifold of scalar curvature $scal=-2n(2n+1)(\alpha^2+\beta^2)$.
\end{thm}
\begin{proof}
From the identity $R(X,Y)\widetilde{P}=0$, we get
\begin{equation}\label{e41}
\widetilde{P}(R(X,Y)U,V)+\widetilde{P}(U,R(X,Y)V)=0.
\end{equation}
Putting $X=U=\xi$ and using \eqref{e5} and \eqref{e41} we have
\begin{equation}\label{e42}
-(\alpha^2+\beta^2)(\eta(Y)\widetilde{P}(\xi,V)+g(Y,V)\widetilde{P}(\xi,\xi)-\widetilde{P}(Y,V)-\eta(V)\widetilde{P}(\xi,Y))=0.
\end{equation}
Using \eqref{e41} in \eqref{e42}, we obtain that $Ric(X,Y)=-2n(\alpha^2+\beta^2)g(X,Y)$ and $scal=2n(2n-1)(\alpha^2+\beta^2)$.
\end{proof}
The pseudo-projective curvature tensor is defined by
\begin{equation}\label{e43}
\overline{P}(X,Y)Z=aR(X,Y)Z+b(Ric(Y,Z)X-Ric(X,Z)Y)-
\end{equation}
$$-\frac{(a+2nb)scal}{2n(2n+1)}(g(Y,Z)X-g(X,Z)Y),$$
where $a,b$ are constants such that $a,b\neq 0$.

We have the following
\begin{thm}\label{thm7}
If a trans-para-Sasakian manifold is pseudo-projectively flat, then it is an Einstein manifold and a manifold of scalar curvature $scal=-2n(2n+1)(\alpha^2+\beta^2)$.
\end{thm}
\begin{proof}
Suppose that $\overline{P}(X,Y)Z=0$, then from \eqref{e43}, we get
\begin{equation}\label{e44}
aR(X,Y)Z+b(Ric(Y,Z)X-Ric(X,Z)Y)-
\end{equation}
$$-\frac{(a+2nb)scal}{2n(2n+1)}(g(Y,Z)X-g(X,Z)Y)=0.$$
Taking the inner product on both sides of \eqref{e44} by $\xi$, we get
\begin{equation}\label{e45}
a\eta(R(X,Y)Z)+b(Ric(Y,Z)\eta(X)-Ric(X,Z)\eta(Y))-
\end{equation}
$$-\frac{(a+2nb)scal}{2n(2n+1)}(g(Y,Z)\eta(X)-g(X,Z)\eta(Y))=0.$$
Putting $X=\xi$ and using \eqref{e5} and \eqref{e12} in \eqref{e45}, we get
\begin{equation}\label{e46}
-a(\alpha^2+\beta^2)(g(Y,Z)-\eta(Y)\eta(Z))+b(Ric(Y,Z)+2n(\alpha^2+\beta^2)\eta(Y)\eta(Z))+
\end{equation}
$$+(a+2nb)(\alpha^2+\beta^2)(g(Y,Z)-\eta(Y)\eta(Z))=0.$$
From the identity \eqref{e46}, we obtain that $Ric(X,Y)=-2n(\alpha^2+\beta^2)g(Y,Z)$ and $scal=-2n(2n+1)(\alpha^2+\beta^2)$.
\end{proof}
\begin{thm}\label{thm8}
A trans-para-Sasakian manifold is satisfying the relation $R(X,Y)\overline{P}=0$ is an Einstein manifold and a manifold of scalar curvature $scal=-2n(2n+1)(\alpha^2+\beta^2)$.
\end{thm}
\begin{proof}
From equality \eqref{e43}, we have $\eta(\overline{P}(X,Y)\xi)=0$.
Now,
$$(R(X,Y)\overline{P}(U,V)Z=R(X,Y)\overline{P}(U,V)Z-\overline{P}(R(X,Y)U,V)Z-\overline{P}(U,R(X,Y)V)Z-$$
$$-\overline{P}(U,V)R(X,Y)Z.$$
By assumption $R(X,Y)\overline{P}=0$, so we have
\begin{equation}\label{e47}
R(X,Y)\overline{P}(U,V)Z-\overline{P}(R(X,Y)U,V)Z-\overline{P}(U,R(X,Y)V)Z-
\end{equation}
$$-\overline{P}(U,V)R(X,Y)Z=0.$$
Therefore
$$g(R(\xi,Y)\overline{P}(U,V)Z,\xi)-g(\overline{P}(R(\xi,Y)U,V)Z,\xi)-g(\overline{P}(U,R(\xi,Y)V)Z,\xi)-$$
$$-g(\overline{P}(U,V)R(\xi,Y)Z,\xi)=0.$$
From this, it follows that,
\begin{equation}\label{e48}
-\overline{P}(U,V,Z,Y)+\eta(Y)\eta(\overline{P}(U,V)Z)-\eta(U)\eta(\overline{P}(Y,V)Z)+
\end{equation}
$$+g(Y,U)\eta(\overline{P}(\xi,V)Z)-\eta(V)\eta(\overline{P}(U,Y)Z)+g(Y,V)\eta(\overline{P}(U,\xi)Z)-$$
$$-\eta(Z)\eta(\overline{P}(U,V)Y)=0.$$
Let $\{e_i\}$, $i=1,...,2n+1$ be an orthonormal basis. Then summing up for $1\leq i \leq 2n+1$ of the relation \eqref{e48} for $Y=U=e_i$ yields
\begin{equation}\label{e49}
\overline{P}(e_i,V,Z,e_i)-2n\eta(\overline{P}(\xi,V)Z)+\eta(Z)\eta(\overline{P}(e_i,V)e_i)=0.
\end{equation}
Taking the trace of the identity, we obtain
\begin{equation}\label{e50}
-\overline{P}(e_i,V,Z,e_i)+2n\overline{P}(\xi,V,Z,\xi)+\eta(Z)\overline{P}(\xi,e_i,e_i,\xi)=0.
\end{equation}
From identity \eqref{e50}, we get
\begin{equation}\label{e50.1}
aRic(V,Z)=-2n.a(\alpha^2+\beta^2)g(V,Z)+(b.scal+2n(2n+1)b(\alpha^2+\beta^2))\eta(V)\eta(Z).
\end{equation}
Taking $Z=\xi$ in \eqref{e50.1} and using \eqref{e12} we obtain
\begin{equation}
scal=-2n(2n+1)(\alpha^2+\beta^2)\quad and\quad Ric(V,Z)=-2n(\alpha^2+\beta^2)g(V,Z).
\end{equation}
\end{proof}

The PC-Bochner curvature tensor on $M$ is defined by \cite{Z1}
$$\mathbf{B}(X,Y,Z,W)=R(X,Y,Z,W)+\frac{1}{2n+4}(Ric(X,Z)g(Y,W)-Ric(Y,Z)g(X,W)+$$
$$+Ric(Y,W)g(X,Z)-Ric(X,W)g(Y,Z)+Ric(\varphi X,Z)g(Y,\varphi W)-$$
$$-Ric(\varphi Y,Z)g(X,\varphi W)+Ric(\varphi Y,W)g(X,\varphi Z)-Ric(\varphi X,W)g(Y,\varphi Z)+$$
$$+2Ric(\varphi X,Y)g(Z,\varphi W)+2Ric(\varphi Z,W)g(X,\varphi Y)-Ric(X,Z)\eta(Y)\eta(W)+$$
$$+Ric(Y,Z)\eta(X)\eta(W)-Ric(Y,W)\eta(X)\eta(Z)+Ric(X,W)\eta(Y)\eta(Z))+$$
$$+\frac{k-4}{2n+4}(g(X,Z)g(Y,W)-g(Y,Z)g(X,W))-\frac{k+2n}{2n+4}(g(Y,\varphi W)g(X,\varphi Z)-$$
$$-g(X,\varphi W)g(Y,\varphi Z)+2g(X,\varphi Y)g(Z,\varphi W))-\frac{k}{2n+4}(g(X,Z)\eta(Y)\eta(W)-$$
$$-g(Y,Z)\eta(X)\eta(W)+g(Y,W)\eta(X)\eta(Z)-g(X,W)\eta(Y)\eta(Z)),$$
where $k=-\frac{scal-2n}{2n+2}$.

Using the PC-Bochner curvature tensor we have
\begin{thm}\label{thm9}
If a trans-para-Sasakian manifold is para-contact conformally flat, then $\alpha^2+\beta^2=1$.
\end{thm}
\begin{proof}
Suppose that the manifold is para-contact conformally flat. Then the condition $\mathbf{B}(X,Y)Z=0$ holds.
Putting $X=Z=\xi$ and using \eqref{e6}, we obtain
\begin{equation}\label{e51}
(\alpha^2+\beta^2-1)(Y-\eta(Y)\xi)=0.
\end{equation}
Since $Y-\eta(Y)\xi=\varphi^2Y\neq 0$, we have $\alpha^2+\beta^2-1=0$.
\end{proof}
\begin{thm}\label{thm10}
If a trans-para-Sasakian manifold satisfies the condition $\mathbf{B}(\xi,Y)Ric=0$, then it is either an Einstein manifold with scalar curvature $scal=-2n(2n+1)(\alpha^2+\beta^2)$ or $\alpha^2+\beta^2=1$.
\end{thm}
\begin{proof}
Suppose that the condition $\mathbf{B}(\xi,Y)Ric(Z,V)=0$ holds.This condition implies that
\begin{equation}\label{e52}
Ric(\mathbf{B}(\xi,Y)Z,V)+Ric(Z,\mathbf{B}(\xi,Y)V)=0.
\end{equation}
Putting $V=\xi$ and using \eqref{e6}, we obtain
\begin{equation}\label{e53}
(\alpha^2+\beta^2-1)(Ric(Y,Z)+2n(\alpha^2+\beta^2)g(Y,Z))=0.
\end{equation}
\end{proof}

\section*{Acknowledgments}

S.Z. is partially supported by Contract  DN 12/3/12.12.2017 and Contract 80-10-24/17.04.2018 with the Sofia University "St.Kl.Ohridski".

\end{document}